\documentclass[11pt,reqno]{amsart}
\usepackage[margin=1.15in]{geometry}
\usepackage{amsmath,amssymb,amsthm,mathrsfs}
\usepackage{tikz}
\usepackage[colorlinks=true,linkcolor=blue,citecolor=blue]{hyperref}
\usetikzlibrary{arrows.meta}
\usepackage{tikz-cd}
\theoremstyle{plain}
\newtheorem{theorem}{Theorem}[section]
\newtheorem{proposition}[theorem]{Proposition}
\newtheorem{lemma}[theorem]{Lemma}
\newtheorem{corollary}[theorem]{Corollary}
\theoremstyle{definition}

\newtheorem{example}[theorem]{Example}

\newtheorem{definition}[theorem]{Definition}
\newcommand{\Spec}{\operatorname{Spec}}
\newcommand{\hgt}{\operatorname{ht}}
\newcommand{\Sing}{\operatorname{Sing}}
\newcommand{\Ann}{\operatorname{Ann}}
\newcommand{\Frac}{\operatorname{Frac}}

\newcommand{\Rt}{\widetilde{R}}
\newcommand{\It}{\widetilde{I}}
\newcommand{\cc}{\mathfrak{c}}
\newcommand{\CC}{\mathfrak{C}}
\newcommand{\mm}{\mathfrak{m}}
\newcommand{\pp}{\mathfrak{p}}

\begin{document}

\title[Nori's question and the conductor]{On a question of Nori: the high-codimension conductor case}

\author{Jebasingh R}
\date{\today}
\address{Department of Mathematics, Indian Insitute of Technology Madras, Chennai, Tamil Nadu 600036}

\begin{abstract}
Let $R$ be an affine domain of dimension $d$ over an infinite field, with normalisation
$\Rt$ smooth, and let $\cc$ be the conductor. We answer a question of Nori
 in the affirmative when $R$ is singular only in high codimension.
\end{abstract}

\maketitle

\section{Introduction}

Let $A$ be a smooth affine domain of dimension $d$ over an infinite (perfect) field $k$, and let
$I\subseteq A[T]$ be an ideal of height $n$ with $\mu(I/I^2T)=n$, where $2n\ge d+3$. Suppose $I=(f_1,\dots,f_n)+I^2T$, Nori asked \cite[Appendix]{Man92} whether there exist
$F_i\in I$ with
\[
I=(F_1,\dots,F_n),\qquad F_i-f_i\in I^2T .
\]
 Mandal \cite{Man92} gave an affirmative answer, with no smoothness
hypothesis, whenever $I$ contains a monic polynomial; that case is excluded throughout what follows.
The local case was settled affirmatively by Mandal--Varma \cite{MV97}, by
Bhatwadekar--Sridharan \cite{BS98} for $n=d\ge 3$, and in general by Bhatwadekar--Keshari
\cite{BK03}.

Smoothness is not a technicality. Bhatwadekar--Sridharan record an example of dimension three normal $\mathbb{C}$-algebra of with an isolated singularity, and an ideal $I\subseteq R[T]$ of height three, for which
the lifting fails. It is therefore natural to ask what smoothness is actually being used for,
and to look for hypothesis on a singular $R$ under which the answer remains affirmative.

There are two ways to do this. One may constrain the ideal, keeping it away from the
singularities: this is the route of Banerjee--Das \cite[Theorem 4.5]{BD}, who require
that $I\cap R$ be contained only in smooth maximal ideals of $R$, so that localising at
$1+(I\cap R)$ produces a regular ring to which \cite{BK03} applies. Or one may constrain
the ring which we do here. Let
$\widetilde{R}$ be the normalisation of $R$, finite over $R$, and let $\mathfrak{c}$ be
the conductor of $R\subset\widetilde{R}$. When $\widetilde{R}$ is smooth, the conductor
cuts out the singular locus exactly, $V(\mathfrak{c})=\operatorname{Sing}(R)$. Our main result is that Nori's question has an affirmative answer when $R$ is singular only in high codimension.

\begin{theorem}\label{thm:main-intro}
Let $R$ be an affine domain of dimension $d$ over an infinite perfect field $k$ whose
normalisation $\Rt$ is smooth over $k$; let $\cc$ be the conductor ideal of $R\subset \Rt$. Let $I\subseteq R[T]$ be an ideal
of height $n$ with $\mu(I/I^2T)=n$ and $2n\ge d+3$, and suppose $I=(f_1,\dots,f_n)+I^2T$ is
given. If $\hgt\cc\ge d-n+2$, then $I=(g_1,\ldots,g_n)$ with $g_i-f_i\in I^2T$.

\end{theorem}
As a consequence of \ref{thm:main-intro} we show the $\mathbb{A}^1$-invariance of the Euler class groups for singular affine varieties (\ref{thm:euler}).
We emphasise that nothing is asked of the
ideal. In particular $V(I\cap R)$ may pass through the singularities of $R$ (Example~\ref{ex:hyperbola}). Theorem~\ref{thm:main-intro} does not reduce to a condition on $I\cap R$, and in particular is
not obtained by applying \cite[Theorem 4.5]{BD} to the ideals arising in its proof. The two
hypotheses are in fact independent: for $R=k[t^2,t^3][y_1,\dots,y_{d-1}]$ the conductor has
height one, so ours fails while that of \cite{BD} holds for any $I$ with $V(I\cap R)$
contained in $\Spec(R[1/t])$.

The condition $\hgt\cc\ge 2$
forces $R$ to satisfy Serre's $R_1$ while failing $S_2$; in particular no
Cohen--Macaulay ring satisfies the hypothesis (\ref{cor:notCM}). The examples are the pinchings: $R=\{f\in\mathcal{O}(X)\mid
f|_{Z_0}=f|_{Z_1}\circ\varphi\}$ for disjoint smooth closed subvarieties $\varphi: Z_0\simeq Z_1$ of a
smooth affine $X$, with $\dim Z_i\le n-2$.

\subsection*{Conventions}
All rings are commutative Noetherian and all modules finitely generated. For a module $M$,
$\mu(M)$ denotes the minimal number of generators. Throughout, $k$ is an infinite field, $R$
is an affine domain of dimension $d$ over $k$, $\Rt$ is the integral closure of $R$ in $\Frac(R)$.  ``Smooth'' means smooth over $k$.

\section{Preliminaries}

  \begin{lemma}\label{lem:height}
    Let $R\subset S$ be finite extension of affine domains over a field $k$. Let $I$ be an ideal of $R$. Then $\hgt(I)=\hgt(IS)$. 
\end{lemma}
\begin{proof}
    We first show that for a prime ideal $\mathfrak{q}$ of $S$, $\hgt(\mathfrak{q}\cap R)=\hgt(\mathfrak{p})$. By incomparibility of integral extension one has $\hgt(\mathfrak{q}\cap R)\geq \hgt(\mathfrak{p})$. As $R,S$ are affine domains over $k$ one has $\dim(S/\mathfrak{q})+\hgt(\mathfrak{q})=\dim(S)$, as $R\subset S$ and $R/(\mathfrak{q}\cap R)\subset  S/\mathfrak{q}$ are integral extensions one has $\dim(S/\mathfrak{q})=\dim(R/(\mathfrak{q}\cap R)$ and $\dim(R)=\dim(S)$. Therefore $\hgt(\mathfrak{q})=ht(\mathfrak{q}\cap R)$.

    Let $\mathfrak{q}$ be minimal prime over $IS$. Then $I\subset \mathfrak{q}\cap R$. As $\hgt(\mathfrak{q}\cap R)=\hgt(\mathfrak{q})$. It follows $\hgt(I)\leq \hgt(IS)$. Let $\mathfrak{p}$ be prime ideal of $R$ containing $I$ such that $\hgt(I)=\hgt(\mathfrak{p})$. Let $\mathfrak{q}$ be prime ideal lying over $\mathfrak{p}$. Then $IS\subset (\mathfrak{q}\cap R)S\subset \mathfrak{q}$ and $\hgt(\mathfrak{q)}=\hgt(\mathfrak{p)}$. It follows that $\hgt(IS)\leq \hgt(I).$ 
   \end{proof}  
   \begin{lemma}\label{lem:square}
$B=\Rt[T]$ is the normalisation of $A=R[T]$, and $\Ann_A B/A=\cc[T]=\CC$. In particular
$\hgt\CC=\hgt\cc$.
\end{lemma}

\begin{proof}
$B$ is normal, integral over $A$, with $\Frac(B)=\Frac(A)$. As $A$-modules $B/A=(\Rt/R)[T]$,
and $\Ann_{R[T]}(M[T])=\Ann_R(M)R[T]$ for finitely generated $M$.
\end{proof}
\begin{lemma}\label{lem:loc}
For every $\pp\in\Spec R$ one has $\cc\not\subseteq\pp$ if and only if $R_\pp$ is normal.
\end{lemma}

\begin{proof}
Integral closure commutes with localisation, so $(\Rt)_\pp=\widetilde{R_\pp}$. Since $\Rt/R$
is a finitely generated $R$-module, $\Ann_R(\Rt/R)_\pp=\Ann_{R_\pp}\big((\Rt/R)_\pp\big)$.
Hence $\cc\not\subseteq\pp$ iff $(\Rt/R)_\pp=0$ iff $R_\pp=\widetilde{R_\pp}$.
\end{proof}
\begin{proposition}\label{prop:dictionary}
Assume $\Rt$ is smooth over $k$. Then for every $\pp\in\Spec R$ the following are equivalent:
$\cc\not\subseteq\pp$; $R_\pp$ is normal; $R_\pp$ is regular; $R_\pp$ is smooth over $k$.
Consequently $V(\cc)=\Sing(R)$.
\end{proposition}

\begin{proof}
The first two are Lemma~\ref{lem:loc}. If $\cc\not\subseteq\pp$ then $R_\pp=(\Rt)_\pp$ is a
localisation of a smooth $k$-algebra, hence essentially smooth, hence regular. Finally a
regular local ring is a unique factorisation domain and therefore integrally closed.
\end{proof}
\begin{corollary}\label{cor:BDdict}
Assume $\Rt$ is smooth and let $J\subseteq R$ be a proper ideal. Then $J$ is contained only in
smooth maximal ideals of $R$ if and only if $\cc+J=R$, if and only if $R_{1+J}$ is essentially
smooth over $k$. 
\end{corollary}

\begin{proof}
If $\cc+J\ne R$ choose $\mm\supseteq\cc+J$ maximal; then $R_\mm$ is not smooth by
Proposition~\ref{prop:dictionary}. Conversely,  $\Spec R_{1+J}=\{\pp:\pp+J\ne R\}$, and each
such $\pp$ lies under a maximal ideal containing $J$; if $\cc+J=R$ then $\cc\not\subseteq\pp$
for all of them, and $R_\pp$ is smooth. The last equivalence follows since
$\cc+J=R$ gives $\cc R_{1+J}=R_{1+J}$, whence $R_{1+J}=\Rt_{1+J\Rt}$ is a localisation of a
smooth ring.
\end{proof}

\begin{proposition}\label{prop:S2}
Let $R$ be a non-normal affine domain satisfying Serre's condition $S_2$. Then every minimal
prime of $\cc$ has height one; in particular $\hgt\cc=1$ and $\dim V(\cc)=d-1$.
\end{proposition}

\begin{proof}
Let $\pp$ be minimal over $\cc$ and suppose $\hgt\pp\ge 2$. Every prime $\mathfrak
q\subsetneq\pp$ satisfies $\cc\not\subseteq\mathfrak q$, so $R_{\mathfrak q}$ is normal by
Lemma~\ref{lem:loc}; for $\hgt\mathfrak q=1$ this makes $R_{\mathfrak q}$ a discrete valuation
ring. Thus $R_\pp$ satisfies $R_1$, and it satisfies $S_2$; by Serre's criterion $R_\pp$ is
normal, contradicting $\cc\subseteq\pp$. Since $\cc\ne 0$ and $R$ is a domain, $\hgt\pp=1$.
\end{proof}

\begin{corollary}\label{cor:notCM}
Under the hypotheses of Theorem~\ref{thm:main-intro} with $n\le d$, the ring $R$ satisfies
$R_1$ and fails $S_2$; in particular $R$ is not Cohen--Macaulay.
\end{corollary}

\begin{proof}
$\hgt\cc\ge d-n+2\ge 2$, so $R_\pp$ is normal, hence regular, for every $\pp$ of height one:
this is $R_1$. If $R$ were $S_2$ then $\hgt\cc=1$ by Proposition~\ref{prop:S2} (note $R$ is
not normal, else $\cc=R$ and $R$ is smooth), a contradiction.
\end{proof}
The following is a lemma of Mohan Kumar \cite[Lemma 3.2]{BK03}.
\begin{lemma}\label{lem:mohan}
Let $R$ be a ring and $J\subset R$ be an ideal of $R$. Let $K\subset J$
and $L\subset J^2$ be two ideals of $R$ such that $K+L=J$. Then $J=K+(e)$
for some $e\in L$ with $e(1-e)\in K$
and $K=J\cap J'$, where $J'+L=R$.
\end{lemma}

\begin{lemma}[Moving lemma]\label{lem:moving-a'}
Let $R$ be a ring and let $P$ be a projective $R$-module of rank $n$. Let
$C,J\subset R$ be two ideals and suppose that
$\alpha\colon P/JP\twoheadrightarrow J/(J^{2}C)$ is a surjection. Let
$\mathfrak{a}\subset R$ be an ideal with $\dim(R/\mathfrak{a})\le n-1$. Then
there exist an ideal $J'\subset R$ and a surjection
$\beta\colon P\twoheadrightarrow J\cap J'$ such that:
\begin{enumerate}
  \item $(J\cap C\mathfrak{a})+J'=R$; in particular $J'+J=R$, $J'+C=R$ and
        $J'+\mathfrak{a}=R$;
  \item $\beta\otimes R/(J^{2}C)=\alpha$;
  \item $\operatorname{ht}(J')\ge n$.
\end{enumerate}
\end{lemma}

\begin{proof}
As $P$ is projective, $\alpha$ lifts to an $R$-linear map $\beta_{0}\colon P\to J$,
and then $\beta_{0}(P)+(J^{2}C)=J$.

\smallskip
\noindent\emph{Step 1.}
Let bar denote reduction modulo $\mathfrak{a}$ and put
$L=(J^{2}C+\mathfrak{a})/\mathfrak{a}\subseteq\bar{J}^{2}$. Then
$\bar{\beta_{0}}\colon\bar{P}\twoheadrightarrow\bar{J}/L$ is a surjection and
$\bar{P}$ is a projective $\bar{R}$-module of rank $n$ with
\[
  n\ \ge\ \dim\bar{R}+1\ \ge\ \dim\bigl(\bar{R}/\mathcal{J}(\bar{R})\bigr)+1 .
\]
By \cite[Lemma 4.4]{BK03} there is a surjection
$\bar{\Psi}\colon\bar{P}\twoheadrightarrow\bar{J}$ lifting $\bar{\beta_{0}}$, so
that $(\bar{\Psi}-\bar{\beta_{0}})(\bar{P})\subseteq L$. As $P$ is projective and
$J^{2}C$ surjects onto $L$, we may choose
$\delta\in\operatorname{Hom}_{R}(P,J^{2}C)$ lifting $\bar{\Psi}-\bar{\beta_{0}}$
and replace $\beta_{0}$ by $\beta_{0}+\delta$. This alters $\beta_{0}$ only
modulo $J^{2}C$, so it is still a lift of $\alpha$ with
$\beta_{0}(P)+(J^{2}C)=J$, and, writing $K=\beta_{0}(P)$, we now have
\begin{equation}\label{eq:onto-mod-a}
  K+(J\cap\mathfrak{a})=J .
\end{equation}

\smallskip
\noindent\emph{Step 2: $\alpha$ lifts to a surjection
$\alpha'\colon P/JP\twoheadrightarrow J/(J^{2}C\mathfrak{a})$.}
By \ref{lem:mohan}, there exists $b_{0}\in J^{2}C$ with $K+(b_{0})=J$. Since
$b_{0}\in J^{2}=(K+(b_{0}))^{2}\subseteq K+(b_{0}^{2})$, we may write
$b_{0}=\kappa_{0}+cb_{0}^{2}$ with $\kappa_{0}\in K$ and $c\in R$. Put
$b_{1}=cb_{0}\in J^{2}C$. Then $b_{0}=\kappa_{0}+b_{1}b_{0}$, so $K+(b_{1})=J$,
and
\[
  b_{1}-b_{1}^{2}=c\bigl(b_{0}-cb_{0}^{2}\bigr)=c\kappa_{0}\in K .
\]
As $b_{1}\in J$, \eqref{eq:onto-mod-a} allows us to write $b_{1}=\kappa+a$ with
$\kappa\in K$ and $a\in J\cap\mathfrak{a}$. Set $b=a\,b_{1}$. Then
\[
  b=b_{1}^{2}-\kappa b_{1}= b_{1}^{2}= b_{1} \pmod K ,
\]
the last congruence because $b_{1}-b_{1}^{2}\in K$; hence $K+(b)=K+(b_{1})=J$.
On the other hand $b=a\,b_{1}\in\mathfrak{a}\cdot J^{2}C=J^{2}C\mathfrak{a}$.
Therefore
\[
  \beta_{0}(P)+(J^{2}C\mathfrak{a})=K+(b)=J ,
\]
so $\beta_{0}$ induces a surjection
$\alpha'\colon P/JP\twoheadrightarrow J/(J^{2}C\mathfrak{a})$, which lifts
$\alpha$ since $J^{2}C\mathfrak{a}\subseteq J^{2}C$.

\smallskip
\noindent\emph{Step 3.}
Apply \cite[Lemma 2.6]{JZ} to $\alpha'$, with the ideal $C\mathfrak{a}$ in place
of $C$. We obtain an ideal $J'\subset R$ and a surjection
$\beta\colon P\twoheadrightarrow J\cap J'$ with $(J\cap C\mathfrak{a})+J'=R$ and
$\operatorname{ht}(J')\ge n$, and such that $\beta$ is a lift of $\alpha'$. As
$\alpha'$ lifts $\alpha$, assertion (2) follows. Finally
$J\cap C\mathfrak{a}$ is contained in each of $J$, $C$ and $\mathfrak{a}$.
\end{proof}

We recall the notion  of analytic isomorphism and a result of Nashier \cite[Proposition 1.3]{Nas83}.

\begin{definition}\label{def:analytic}
Let $A\subseteq B$ be an extension of rings and $0\neq f\in A$ a non-zero divisor in $A$. We say that $A\subseteq B$ is an
\emph{analytic isomorphism along $f$} if
\[
A/(f)\;\xrightarrow{\ \sim\ }\;B/(f),
\]
or equivalently if $B=A+fB$ and $fB\cap A=fA$.
\end{definition}

It follows at once from the definition that if $A_1\subseteq A$ is an analytic isomorphism
along $f$, then it is an analytic isomorphism along $f^m$ for every positive integer $m$.

\begin{proposition}\label{prop:nashier}
Let $A$ be a subring of a ring $B$, let $I\subseteq B$ be an ideal and set $I_1=I\cap A$.
Suppose there is an element $f\in I_1$ such that $A\subseteq B$ is an analytic isomorphism
along $f$. Then the following holds
\begin{enumerate}
\item $A/I_1\simeq B/I$;
\item $I=I_1A$;
\item $I_1/I_1^2\simeq I/I^2$.
\end{enumerate}
\end{proposition}

\section{Main results}
\begin{theorem}\label{thm:main}
Let $R$ be an affine domain of dimension $d$ over an infinite perfect field $k$ whose
normalisation $\Rt$ is smooth over $k$; let $\cc$ be the conductor ideal of $R\subset \Rt$. Let $I\subseteq R[T]$ be an ideal
of height $n$ with $\mu(I/I^2T)=n$ and $2n\ge d+3$, and suppose $I=(f_1,\dots,f_n)+I^2T$ is
given. If $\hgt\cc\ge d-n+2$, then $I=(h_1,\ldots,h_n)$,  $h_i-f_i\in I^2T$.

\end{theorem}
\begin{proof}
    Let $A=R[T]$, $B=\Rt[T]$. 
Applying moving lemma~\ref{lem:moving-a'} we get $I_0$ ideal of $A$ such that
\begin{equation}\label{eq:move}
I\cap I_0=(b_1,\dots,b_n),\quad b_i-f_i\in I^2T,\quad I+I_0=A,\quad I_0+\CC=A,\quad
I_0+I\cap(T)=A,
\end{equation}
and $I_0=A$ or $\hgt I_0=n$. In the first case $I=(b_1,\dots,b_n)$ and we are done; assume
$\hgt I_0=n$. Let $\It_0=I_0B$. By ~\ref{lem:height},
$\hgt\It_0=n$. 
From $I+I_0=A$ we get $I\cap I_0=II_0$, so
\[
(b_1,\dots,b_n)+I_0^2=II_0+I_0^2=I_0(I+I_0)=I_0,
\] 
 Moreover $I\cap (T)\subseteq TA$, so \eqref{eq:move} gives $I_0+TA=A$
and therefore $I_0(0)=R$. By \cite[Remark 3.9]{BS98} a surjection $A^n\twoheadrightarrow
I_0/I_0^2$ lifts to $A^n\twoheadrightarrow I_0/I_0^2T$ as if the induced surjection at
$T=0$ lifts to a surjection $R^n\twoheadrightarrow I_0(0)$, which is vacuous here. 
Hence
\begin{equation}\label{eq:Tpres}
I_0=(b_1',\dots,b_n')+I_0^2T,\qquad b_i'-b_i\in I_0^2 .
\end{equation}
Since $\Rt$ is a smooth affine $k$-domain
of dimension $d$, $k$ is infinite perfect, and $2n\ge d+3$, \cite[Theorem 4.13]{BK03} applied
to $\It_0\subseteq B$ yields
\begin{equation}\label{eq:up}
\It_0=(\widetilde F_1,\dots,\widetilde F_n),\qquad \widetilde F_i-b_i'\in\It_0^2T .
\end{equation}

Note $\It_0\cap A=I_0$.  Let $x\in \It_0\cap A$. We have, by (\ref{eq:move}), $u+c=1$ for some $u\in I_0$ and $c\in \CC$. This gives $xu+xc=x$. As $x\in A$ and $u\in I_0$, we have $xu\in I_0$. Also, $xc\in \CC \It_0= I_0A=I_0.$ This proves the equality.\\
 As $\It_0+\CC=B$, we have $I_0+\CC=A$. Hence there exists $c\in \CC$ such that $l=1-c\in I_0$. We may assume $ht(l)=1$. If not choose $l'$ which is not in any minimal prime of $B$, then we have $ht(l+l'-ll')=1$. Let $l''=l+l'=l''$, then we have $1-l''\in C$ and $ht(l'')=1$ and we may replace $l$ with $l''$. Therefore we have $A\subset B$ is an analytic isomorphism along $l\in I_0$. By \ref{prop:nashier}
\begin{equation}
    A/I_0\simeq B/I_0B,
     \qquad l\in I_0\text{ gives } I_0/I_0^2\simeq \It_0/\It_0^2
\end{equation}
Let $a_1,\ldots, a_n$ be the generators of $I_0$ corresponding to $\widetilde F_1,\ldots,\widetilde F_n$. Then we have
$  a_i-b_i\in I_0^2 $.
Let $R(T)$ denote the ring obtained by inverting all the monic polynomials in $R[T]$. Then $\dim(R(T))=d$ and applying subtraction principle \cite[Proposition 3.1]{BS00-1} we have  
\begin{equation}
    IR(T)= (c_1,\ldots, c_n), \qquad \text{with } c_i-b_i\in I^2R(T)
\end{equation}
It follows $c_i-f_i\in I^2R(T)$. Moreover, by
 \cite[Proposition 7.4]{KT23} we have $I=(h_1,\ldots,h_n)$ with $h_i-f_i\in I^2T$.

\end{proof}
\begin{example}\label{ex:hyperbola}
Take $d=5$, $n=4$ (so $2n=d+3$), $\Rt=k[x_1,\dots,x_5]$, $p=(0,0,0,0,0)$,
$q=(1,0,0,0,0)$, and $R=\{f\in\Rt\mid f(p)=f(q)\}$.

 \[\begin{tikzcd}[ampersand replacement=\&,cramped]
	R \& k \\
	{\tilde{R}=k[X_1,\cdots,X_5]} \& {\tilde{R}/\mathfrak{m}_p\times\tilde{R}/\mathfrak{m}_q=k\times k}
	\arrow[from=1-1, to=1-2]
	\arrow[from=1-1, to=2-1]
	\arrow[from=1-2, to=2-2]
	\arrow[from=2-1, to=2-2]
\end{tikzcd}\]    
We have $\cc=\mm_p\cap\mm_q$, a maximal ideal of $R$ with
$R/\cc=k$; thus $\hgt\cc=5\ge d-n+2=3$ and Theorem~\ref{thm:main} applies. Explicitly
$R=k[x_2,\dots,x_5,\;x_1^2-x_1,\;x_1x_2,\dots,x_1x_5]$.  note $x_2,x_3,x_4,x_5\in\cc$ and let
\[
I_0=(x_2T-1,\;x_3,\;x_4,\;x_5)\subseteq R[T],
\]
an ideal of height $4$. Since $x_2\in\cc$ we have $1=-(x_2T-1)+T x_2$, so $I_0+\CC=A$:
the ideal $I_0$ is exactly of the kind \ref{thm:main} produces. However
$I_0\Rt[T]\cap\Rt=(x_3,x_4,x_5)$, and one checks $I_0\cap
R=(x_3,x_4,x_5,x_1x_3,x_1x_4,x_1x_5)R$, a prime of height three contained in $\cc$. Hence
\[
(I_0\cap R)+\cc=\cc\ne R,
\]
so $I_0\cap R$ lies in the singular maximal ideal of $R$.
\end{example}

We now derive interesting consequences of \ref{thm:main}.
\begin{theorem}\label{thm:euler}
    Let $R$ be a affine domain of dimension $d\ge3$  over infinite perfect field $k$ of characteristic zero whose normalisation is smooth, with $\hgt\cc\ge 2$. Then the natural map $E^d(R)\xrightarrow[]{}E^d(R[T])$ is an isomorphism.
\end{theorem}

\begin{proof}
Let $(I,\omega_I)\in E^d(R[T])$ where $I$ is an ideal of $R[T]$ of height
$d$ and $\omega_I$ is a local orientation of $I$. By \cite[Lemma 2.9]{Das03}, without
loss of generality we may assume that either $I(0)=A$ or
$\operatorname{ht}(I(0))=n$. If $I(0)=R$, then we can lift $\omega_I$ to a set of generators of
$I/(I^2T)$ and hence $(I,\omega_I)=0$ by \ref{thm:main}.

Now suppose that $\operatorname{ht}(I(0))=n$. We consider the element $(I(0),\omega_I(0))\in E^d(R)$,
induced by $(I,\omega_I)$. Applying \cite[Lemma 2.12]{Das03}, we can find an ideal $K\subset R$ of height $n$ which is comaximal with $I\cap R$ and a
local orientation $\omega_K$ of $K$ such that $(I(0),\omega_I(0))+(K,\omega_K)=0
\quad\text{in }E^d(R)$. Let $L=I\cap K[T]$. Since the ideals $I$ and $K[T]$ are comaximal,
$\omega_I$ and $\omega_K$ induce $\omega_L:(R[T]/L)^n\longrightarrow L/L^2
$ and we have the following equation in $E^d(R[T])$:
\[
(L,\omega_L)
=
(I,\omega_I)+(K[T],\omega_K\otimes R[T]).
\]

Since
\[
(L(0),\omega_L(0))
=
(I(0),\omega_I(0))+(K,\omega_K)
=0
\]
in $E^d(R),$ it follows that we can lift $\omega_L$ to a set of
generators of $L/(L^2T)$. 
By \ref{thm:main}, we have $(L,\omega_L)=0$, proving $\Phi$ is a surjection. 
\end{proof}
The following is an immediate consequence of \cite[Theorem 2.4]{Das18}.
\begin{corollary}\label{cor:finitefield}
Let $R$ be a affine domain of dimension $d\ge3$  over infinite field $k=\overline{k}$  of characteristic $p\neq 0$ whose
normalisation is smooth, with $\hgt\cc\ge 2$. Then $E^d(R[T])=0$.
\end{corollary}
As a consequence of \cite[Corollary 4.11]{Das18} and \cite[Corollary 4.4]{BS00}, we obtain the following.
\begin{corollary}\label{cor:unimod}
Let $R$ be a affine domain of dimension $d\ge3$  over infinite perfect field $k$  of characteristic zero whose
normalisation is smooth, with $\hgt\cc\ge 2$. Let $P$ be a projective $R[T]$-module of rank $d$
with trivial determinant such that $P/TP$ has a unimodular element. Then
$P$ has a unimodular element.
\end{corollary}
\section*{Acknowledgments}
The author acknowledges the financial support provided by the Institute Post-Doctoral Fellowship, Indian Institute of Technology Madras, during the course of this research.

\end{document}